\documentclass[11pt]{article}
\usepackage[utf8]{inputenc}
\usepackage{amsmath}
\usepackage{amsthm}
\usepackage{amssymb}
\usepackage{enumerate}
\usepackage{hyperref}
\usepackage{aliascnt}
\usepackage{graphicx}
\usepackage{subfig}
\usepackage{color}

\theoremstyle{plain}
\newtheorem{thm}{Theorem}[section]

\newaliascnt{lem}{thm}
\newtheorem{lem}[lem]{Lemma}
\aliascntresetthe{lem}

\newaliascnt{pro}{thm}
\newtheorem{pro}[pro]{Proposition}
\aliascntresetthe{pro}

\newaliascnt{cor}{thm}
\newtheorem{cor}[cor]{Corollary}
\aliascntresetthe{cor}  

\newaliascnt{que}{thm}

\aliascntresetthe{que}

\newaliascnt{clm}{thm}
\newtheorem{clm}[clm]{Claim}
\aliascntresetthe{clm}

\newaliascnt{obs}{thm}
\newtheorem{obs}[obs]{Observation}
\aliascntresetthe{obs}

\theoremstyle{definition}
\newaliascnt{exm}{thm}

\aliascntresetthe{exm}

\theoremstyle{plain}

\newcommand{\N}{\mathbb{N}}
\newcommand{\R}{\mathbb{R}}

\newcommand{\I}{\mathcal{I}}
\newcommand{\B}{\mathcal{B}}
\newcommand{\C}{\mathcal{C}}
\renewcommand{\P}{\mathcal{P}}

\DeclareMathOperator{\rank}{rk}
\DeclareMathOperator{\fin}{fin}

\begin{document}

\title{\scshape Infinite matroid union}
\author{Elad Aigner-Horev\footnote{Research supported by the Minerva foundation.} \and Johannes Carmesin \and Jan-Oliver Fröhlich}
\date{University of Hamburg\\9 July 2012}
\maketitle

\begin{abstract}
\noindent
We consider the problem of determining whether the union of two infinite matroids is a matroid.
We introduce a superclass of the finitary matroids, the \emph{nearly finitary matroids}, and prove that the union of two nearly finitary matroids is a nearly finitary matroid.

On the other hand, we prove that the union of two arbitrary infinite matroids is not necessarily a matroid.
Indeed, we show (under a weak additional assumption) that the nearly finitary matroids are essentially the largest class of matroids for which one can have a union theorem. 

We then extend the base packing theorem for finite matroids to finite families of co-finitary matroids.
This, in turn, yields a matroidal proof for the tree-packing results for infinite graphs due to Diestel and Tutte.
\end{abstract}

\section{Introduction}

Recently, Bruhn, Diestel, Kriesell, Pendavingh and Wollan~\cite{matroid_axioms} 
found axioms for infinite matroids in terms of independent sets, bases, circuits, closure and
(relative) rank. These axioms allow for duality of infinite matroids as known
from finite matroid theory, which settled an old problem of Rado. 
With these new axioms it is possible now to look which theorems of finite
matroid theory have infinite analogues.

Here, we shall look at the {\sl matroid union theorem} which is 
a classical result in finite matroid theory ~\cite{Oxley,schrijverBook}.
It says that, given finite matroids $M_1 = (E_1,\I_1)$ and $M_2=(E_2,\I_2)$, 
the set system 
\begin{equation}\label{eq:what-is-union}
\I(M_1\vee M_2) = \{I_1\cup I_2\mid I_1 \in \I_1,\; I_2 \in \I_2\}
\end{equation}
is the set of independent sets of a matroid, the \emph{union matroid} $M_1 \vee M_2$, and specifies a rank function for this matroid.

The matroid union theorem has important applications in finite matroid theory.
For example, it can be used to provide short proofs for the {\sl base covering and packing theorem} (discussed below more broadly), or to the \emph{matroid intersection} theorem~\cite{Oxley}. 

While the union of two finite matroids is always a matroid, it is not true 
that the union of two infinite matroids is always a matroid (see \autoref{exm:fin_union_fail2} below). The purpose of this paper is to study for which matroids their union is a matroid.

\subsection{Our results} In this section, we outline our results with minimal background, deferring details until later sections. 
First we prove the following.

\begin{pro}\label{exm:fin_union_fail2}
If $M$ and $N$ are infinite matroids, then $\I(M_1 \vee M_2)$ is not necessarily a matroid. 
\end{pro}

One of the matroids involved in the proof of this proposition is finitary.  Nevertheless, in \autoref{sec:fin-mat},  we establish a union theorem (see \autoref{thm:nearly-union} below) for a superclass of the finitary matroids which we call {\sl nearly finitary} matroids, defined next.

For any matroid $M$, taking as circuits only the finite circuits of $M$ defines
a (finitary) matroid with the same ground set as $M$. This matroid is called the
\emph{finitarization} of $M$ and denoted by $M^{\fin}$.

It is not hard to show that every basis $B$ of $M$ extends to a
basis $B^{\fin}$ of $M^{\fin}$, 
and conversely every basis $B^{\fin}$ of $M^{\fin}$ contains a basis $B$ of $M$.
Whether or not $B^{\fin}\setminus B$ is finite will in general depend on the
choices for $B$ and $B^{\fin}$, but given a choice for one of the two, it will
no longer depend on the choice for the second one.

We call a matroid $M$ \emph{nearly finitary} if every base of its finitarization contains a base of $M$ such that their difference is finite. 

The class of nearly finitary matroids contains all finitary matroids, but not
only. For example, the set system $\C(M)
\cup \B(M)$ consisting of the circuits of an infinite-rank finitary matroid $M$
together with its bases forms the set of circuits of a nearly finitary matroid
that is not
finitary (see \autoref{thm:fin-to-nearly}).
In \cite{intersection} we characterize the graphic nearly finitary matroids;
this also gives rise to numerous examples of nearly finitary matroids that
are not finitary.

We show that the class of finitary matroids is closed under union (Section 4.2).
In \autoref{sec:nearly-union} we prove the same for the larger class of nearly
finitary
matroids, which is our main result:

\begin{thm}\label{thm:nearly-union}\emph{[Nearly finitary union theorem]}\\
If $M_1$ and $M_2$ are nearly finitary matroids, then $M_1 \vee M_2$ is a
matroid and in fact nearly finitary.
\end{thm}

\noindent \autoref{thm:nearly-union} is essentially best possible as follows.

First, the non-finitary matroid involved in the proof of
\autoref{exm:fin_union_fail2} is a countable direct sum of infinite circuits and
loops. This is essentially the simplest example of a matroid that is not nearly
finitary. 

Second,  we show in \autoref{sec:non-nearly} that for every matroid $N$ that is
not nearly finitary and that satisfies a (weak) additional assumption there
exists a finitary matroid $M$ such that $\I(M\vee N)$ is not a matroid. 
Thus in essence, not only is the class of nearly finitary matroids maximal with
the property of having a union theorem; it is not even possible to add a matroid
that is not nearly finitary to the class of finitary matroids without
invalidating matroid union.

More
precisely, we prove the following counterpart to \autoref{thm:nearly-union}. 

\begin{pro}\label{thm:tightness}
Let $N$ be a matroid that is not nearly finitary. 
Suppose that the finitarization of $N$ has an independent set $I$
containing only countably many $N$-circuits such that $I$ has no finite subset meeting all of these circuits.
Then there exists a finitary matroid $M$ such that $\I(M\vee N)$ is not a matroid.
\end{pro}

A simple consequence of \autoref{thm:nearly-union} is that $M_1 \vee \cdots \vee
M_k$ is a nearly finitary matroid whenever $M_1,\ldots,M_k$ are nearly finitary.
On the other hand, (by \autoref{exm:inf_union_fail}) a countable union of
nearly finitary matroids need not be a matroid.

In finite matroid theory, the {\sl base covering} and {\sl base packing} theorems are two well-known applications of the finite matroid union theorem. The former extends to finitary matroids in a straightforward manner (see \autoref{thm:cover}). 

In \autoref{sec:cover-pack}, we extend the finite base packing theorem to finite families of \emph{co-finitary} matroids; i.e., matroids whose dual is finitary. The finite base packing theorem asserts that {\sl a finite matroid $M$ admits $k$ disjoint bases if and only if $k \cdot \rank(X) +|E(M)\setminus X| \geq k \cdot \rank(M)$ for every $X \subseteq E(M)$}~\cite{schrijverBook}, where $\rank$ denotes the rank function of $M$. For infinite matroids, this rank condition is too crude.  We reformulate this condition using the notion of {\sl relative rank} introduced in~\cite{matroid_axioms} as follows: given two subsets $B \subseteq A \subseteq E(M)$, the \emph{relative rank of $A$ with respect to $B$} is denoted by $\rank(A|B)$, satisfies $\rank(A|B) \in \N \cup \{ \infty \}$, and is given by 
\[
\rank(A|B) = \max \{|I \setminus J|:\; J \subseteq I, \; I \in \I(M) \cap 2^A,\; J \text{ maximal in } \I(M) \cap 2^B\}.
\] 

\begin{thm}\label{thm:pack}
A co-finitary matroid $M$ with ground set $E$ admits $k$ disjoint bases if and only if $|Y|\geq k \cdot \rank(E|E-Y)$ for all finite sets $Y\subseteq E$.
\end{thm}

\autoref{thm:pack} does not extend to arbitrary infinite matroids.
Indeed, for every integer $k$ there exists a finitary matroid with no three disjoint bases and satisfying $|Y|\geq k \cdot \rank(E|E-Y)$ for every $Y \subseteq E$~\cite{aharoniThom,DiestelBook10}.

This theorem gives a short matroidal proof of a result of Diestel and Tutte~\cite[Theorem 8.5.7]{DiestelBook10} who showed that the well-known tree-packing theorem for finite graphs due to Nash-Williams and Tutte \cite{DiestelBook10} extends to infinite graphs with so-called topological spanning trees.

\section{Preliminaries}\label{sec:pre}

Notation and terminology for graphs are that of~\cite{DiestelBook10}, for matroids  that of~\cite{Oxley,matroid_axioms}, and for topology that of~\cite{Armstrong}.

Throughout, $G$ always denotes a graph where $V(G)$ and $E(G)$ denote its vertex and edge sets, respectively. 
We write $M$ to denote a matroid and write $E(M)$, $\I(M)$, $\B(M)$, and $\C(M)$ to denote its ground set, independent sets, bases, and circuits, respectively. 

It will be convenient to have a similar notation for set systems. That is, for a set system $\I$ over some ground set $E$, an element of $\I$ is called \emph{independent}, a maximal element of $\I$ is called a \emph{base} of $\I$, and a minimal element of $\P(E)\setminus \I$ is called \emph{circuit} of $\I$. A set system is \emph{finitary} if an infinite set belongs to the system provided each of its finite subsets does; with this terminology, $M$ is finitary provided that $\I(M)$ is finitary.

We review the definition of a matroid as given in~\cite{matroid_axioms}.
A set system $\I$ is the set of independent sets of a matroid if it satisfies the following \emph{independence axioms}~\cite{matroid_axioms}:
\begin{itemize}
	\item[(I1)] $\emptyset\in \I$.
	\item[(I2)] $\left\lceil \I \right\rceil=\I$, that is, $\I$ is closed under taking subsets.
	\item[(I3)] Whenever $I,I'\in \I$ with $I'$ maximal and $I$ not maximal, there exists an $x\in I'\setminus I$ such that $I+x\in \I$.
	\item[(IM)] Whenever $I\subseteq X\subseteq E$ and $I\in\I$, the set $\{I'\in\I\mid I\subseteq I'\subseteq X\}$ has a maximal element.
\end{itemize}

In~\cite{matroid_axioms}, an equivalent axiom system to the independence axioms is provided and is called the \emph{circuit axioms system}; this axiom system characterises a matroid in terms of its circuits. Of these circuit axioms, we shall make frequent use of the so called \emph{(infinite) circuit elimination axiom} phrased here for a matroid $M$:
\begin{enumerate}
	\item [(C)] Whenever $X\subseteq C\in \C(M)$ and $\{C_x\mid x\in X\} \subseteq \C(M)$ satisfies $x\in C_y\Leftrightarrow x=y$ for all $x,y\in X$, then for every $z \in C\setminus \left( \bigcup_{x \in X} C_x\right)$ there exists a  $C'\in \C(M)$ such that $z\in C'\subseteq \left(C\cup  \bigcup_{x \in X} C_x\right) \setminus X$.
\end{enumerate}

\section{Union of arbitrary infinite matroids}\label{sec:union} 

In this section, we prove \autoref{exm:fin_union_fail2}. 
That is, we show that there exists infinite matroids $M$ and $N$ whose union is not a matroid. 

As the nature of $M$ and $N$ is crucial for establishing the tightness of \autoref{thm:nearly-union}, we prove \autoref{exm:fin_union_fail2} in two steps as follows. 

In \autoref{thm:uncnt-fail}, we treat the relatively simpler case in which $M$ is finitary and $N$ is co-finitary and both have uncountable ground sets. Second, then, in \autoref{clm:cnt-fail}, we refine the argument as to have $M$ both finitary and co-finitary and $N$ co-finitary and both on countable ground sets. 

\begin{clm}\label{thm:uncnt-fail}
There exists a finitary matroid $M$ and a co-finitary matroid $N$ such that $\I(M\vee N)$ is not a matroid. 
\end{clm}

\begin{proof}
Set $E= E(M) = E(N) = \N \times\R$. Next,  
put $M := \bigoplus_{n\in\N} M_n$, where $M_n := U_{1, \{n\} \times \R}$. The matroid $M$ is finitary as it is a direct sum of $1$-uniform matroids. 
For $r \in \R$, let $N_r$ be the circuit matroid on $\N \times \{r\}$; set $N:= \bigoplus_{r\in\R}N_r$. As $N$ is a direct sum of circuits, it is co-finitary.
(see \autoref{fig:fin_union_fail1}).

\begin{figure} [htpb]   
\begin{center}
	\includegraphics{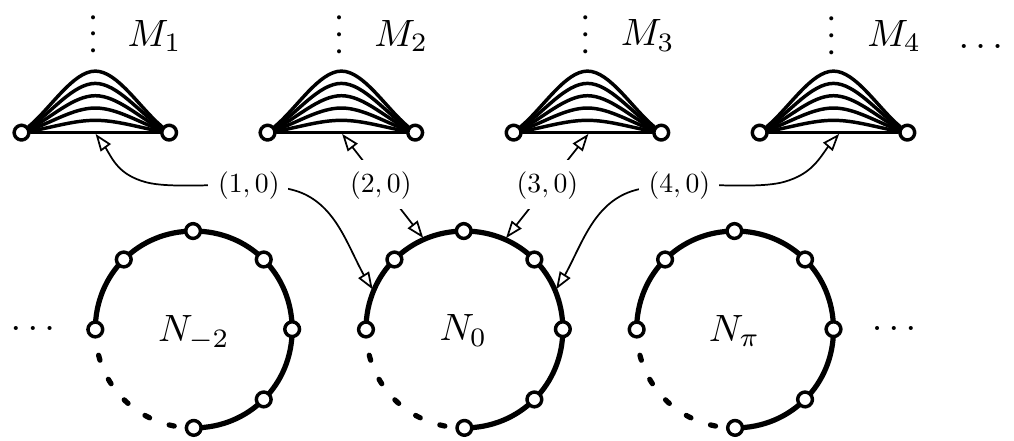}
	\caption{{\small $M=\bigoplus_{n\in\N} M_n$ and $N=\bigoplus_{r\in\R}N_r$.}}
	\label{fig:fin_union_fail1}
\end{center}
\end{figure}

We show that $\I(M\vee N)$ violates the axiom (IM) for $I=\emptyset$ and $X=E$; so that $\I(M \vee N)$ has no maximal elements. 
It is sufficient to show that 
a set $J\subseteq E$ belongs to $\I(M\vee N)$ if and only if it contains at most countably many circuits of $N$. For if so,
then for any $J\in \I(M\vee N)$ and any circuit $C = \N\times \{r\}$ of $N$ with $C\nsubseteq J$ (such a circuit exists) we have $J\cup C\in\I(M\vee N)$.

The point to observe here is that every independent set of $M$ is countable, (since
every such set meets at most one element of $M_n$ for each $n \in \N$), and~that every independent set of $N$ misses uncountably many elements of $E$ (as any such set must miss at least one element of $N_r$ for each $r \in \R$).

Suppose $J\subseteq E$ contains uncountably many circuits of $N$. 
Since each independent set of $N$ misses uncountably many elements of $E$, every set $D = J \setminus J_N$ is uncountable whenever
$J_N \in \I(J)$.  On the other hand, since each independent set of $M$ is countable, we have that $D \notin \I(M)$. 
Consequently, $J\notin \I(M\vee N)$, as required.

We may assume then that $J\subseteq E$ contains only countably many circuits of $N$, namely, $\{C_{r_1},C_{r_2},\ldots\}$. Now the set $J_M = \{(i,r_i): i \in \N\}$ is independent in $M$; consequently, $J \setminus J_M$ is independent in $N$; completing the proof.
\end{proof}

We proceed with matroids on countable ground sets. 

\begin{clm}\label{clm:cnt-fail}
There exist a matroid $M$ that is both finitary and co-finitray, and a co-finitary matroid $N$ whose common ground is countable such that $\I(M\vee N)$ is not a matroid. 
\end{clm}

\begin{proof}
For the common ground set we take $E = (\N\times\N)\cup L$ where $L= \{\ell_1, \ell_2,\ldots\}$ is countable and disjoint to $\N\times \N$.
The matroids $N$ and $M$ are defined as follows. 
For $r\in\N$, let $N_r$ be the circuit matroid on $\N\times\{r\}$. Set $N$ to be the matroid on $E$ obtained by adding the elements of $L$ to  the matroid $\bigoplus_{r\in\N}N_r$ as loops.
Next, 
for $n \in \N$,  let $M_n$ be the $1$-uniform matroid on $ (\{n\}\times\{1,2,\ldots,n\})\cup \{\ell_n\}$.
Let $M$ be the matroid obtained by adding to the matroid $\bigoplus_{n\in\N} M_n$ all the members of $E \setminus E(\bigoplus_{n\in\N} M_n)$ as loops 

We show that $\I(M\vee N)$ violates the axiom (IM) for $I=\N\times \N$ and $X=E$. It is sufficient to show that 

\begin{itemize}
\item [(a)] $I\in\I(M\vee N)$; and that
\item [(b)] every set $J$ satisfying  $I\subset J\subseteq E$ is in $\I(M\vee N)$ if and only if it misses infinitely many elements of $L$.
\end{itemize}

To see that $I\in\I(M\vee N)$, note that the set $I_M = \{(n, n)\ |\ n\in\N\}$ is independent in $M$ and meets each circuit $\N \times \{r\}$ of $N$.
In particular, the set $I_N := (\N\times\N)\setminus I_M$ is independent in $N$, and therefore $I = I_M\cup I_N\in\I(M\vee N)$.

Let then $J$ be a set satisfying $I \subseteq J \subseteq E$, and suppose, first, that $J\in \I(M\vee N)$. We show that $J$ misses infinitely many elements of $L$. 

There are sets $J_M\in\I(M)$ and $J_N\in\I(N)$ such that $J=J_M\cup J_N$.
As $J_N$ misses at least one element from each of the disjoint circuits of $N$ in $I$, the set $D := I\setminus J_N$ is infinite. Moreover, we have that $D\subseteq J_M$, since $I\subseteq J$.
In particular, there is an infinite subset $L' \subseteq L$ such that $D+l$ contains a circuit of $M$ for every $\ell \in L'$. Indeed, for every $e \in D$ is contained in some $M_{n_e}$; let then $L' = \{\ell_{n_e}: e \in D\}$ and note that $L' \cap J = \emptyset$. 
This shows that $J_M$ and $L'$ are disjoint and thus $J$ and $L'$ are disjoint as well, and the assertion follows.

Suppose, second, that there exists a sequence $i_1 < i_2 < \ldots$ such that $J$ is disjoint from $L'=\{\ell_{i_r}\ : \ r \in \N\}$.
We show that the superset $E\setminus L'$ of $J$ is in $\I(M\vee N)$.
To this end, set $D:= \{(i_r, r)\ |\ r\in\N\}$.
Then, $D$ meets every circuit $\N\times \{r\}$ of $N$ in $I$, so that the set  $J_N := \N\times \N\setminus D$ is independent in $N$.
On the other hand, $D$ contains a single element from each $M_n$ with $n\in L'$. Consequently, $J_M := (L\setminus L') \cup D\in\I(M)$ and therefore $E\setminus L' = J_M\cup J_N\in\I(M\vee N)$.
\end{proof}

While the union of two finitary matroids is a matroid, by \autoref{thm:finitary-union}, the same is not true for two co-finitary matroids. 

\begin{cor}
The union of two co-finitary matroids is not necessarily a matroid. 
\end{cor}

\section{Union}\label{sec:matroid-union}

In this section, we prove \autoref{thm:nearly-union}. The main difficulty in proving this theorem is the need to verify that given two nearly finitary matroids $M_1$ and $M_2$, that the set system $\I(M_1 \vee M_2)$ satisfies the axioms (IM) and (I3).

To verify the (IM) axiom for the union of two nearly finitary matroids we shall require the following theorem, proved below in \autoref{sec:fin-mat}. 

\begin{pro}\label{thm:finitary-union} 
If $M_1$ and $M_2$ are finitary matroids, then $M_1 \vee M_2$ is a finitary matroid. 
\end{pro}

To verify (IM) for the union of finitary matroids we use a compactness argument (see \autoref{sec:fin-mat}). More specifically, we will show that $\I(M_1 \vee M_2)$ is a finitary set system whenever $M_1$ and $M_2$ are finitary matroids. It is then an easy consequence of Zorn's lemma that all finitary set systems satisfy (IM).   

The verification of axiom (I3) is dealt in a joint manner for both matroid families.  In the next section we prove the following.

\begin{pro}\label{thm:i3}
The set system $\I(M_1 \vee M_2)$ satisfies \emph{(I3)} for any two matroids $M_1$ and $M_2$. 
\end{pro}  

Indeed, for finitary matroids, \autoref{thm:i3} is fairly simple to prove. We, however, require this proposition to hold for nearly finitary matroids as well. Consequently, we prove this proposition in its full generality, i.e., for any pair of matroids.    
In fact, it is interesting to note that the union of infinitely many matroids satisfies (I3); though the axiom (IM) might be violated as seen in \autoref{exm:inf_union_fail}).

At this point it is insightful to note a certain difference between the union of finite matroids to that of finitary matroids in a more precise manner. By the finite matroid union theorem if $M$ admits two disjoint bases, then the union of these bases forms  a base of $M \vee M$. For finitary matroids the same assertion is false. 

\begin{clm}\label{clm:dif}
There exists an infinite finitary matroid $M$ with two disjoint bases whose union is not a base of the matroid $M\vee M$ as it is properly contained in the union of some other two bases. 
\end{clm}

\begin{proof}
Consider the infinite one-sided ladder with every edge doubled, say $H$, and recall that the bases of $M_F(H)$ are the ordinary spanning trees of $H$.
In \autoref{fig:disj_bases}, $(B_1,B_2)$ and $(B_3, B_4)$ are both pairs of disjoint bases of $M_F(H)$.
However, $B_3\cup B_4$ properly covers $B_1\cup B_2$ as it additionally contains the leftmost edge of $H$
\end{proof}

Clearly, a direct sum of infinitely many copies of $H$ gives rise to an infinite sequence of unions of disjoint bases, each properly containing the previous one. 
In fact, one can construct a (single) matroid formed as the union of two nearly finitary matroids that admits an infinite properly nested sequence
of unions of disjoint bases.
 
\begin{figure} [htpb]   
\begin{center}
	\includegraphics[width=\linewidth]{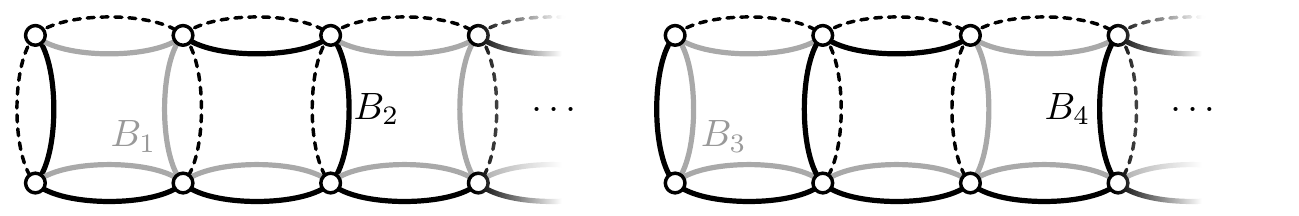}
	\caption{The disjoint bases $B_1$ and $B_2$ on the left are properly covered by the bases $B_3$ and $B_4$ on the right.}
	\label{fig:disj_bases}
\end{center}
\end{figure}

\subsection{Exchange chains and the verification of axiom (I3)}

In this section, we prove \autoref{thm:i3}. Throughout this section $M_1$ and $M_2$ are matroids. It will be useful  to show that the following variant of (I3) is satisfied.

\begin{pro}\label{thm:i3'}
 	The set $\I = \I(M_1\vee M_2)$ satisfies the following. 
	\begin{itemize}
		\item[(I3')] For all $I, B\in\I$ where $B$ is maximal and all $x\in I\setminus B$ there exists $y\in B\setminus I$ such that $(I + y) - x\in\I$.
	\end{itemize}
\end{pro}

\noindent
Observe that unlike in (I3), the set $I$ in (I3') may be maximal. 

We begin by showing that \autoref{thm:i3'} implies \autoref{thm:i3}.

\begin{proof}[Proof of \autoref{thm:i3} from \autoref{thm:i3'}.] 
Let $I\in\I$ be non-maximal and $B\in\I$ be maximal.
As $I$ is non-maximal there is an $x\in E\setminus I$ such that $I+x\in\I$.
We may assume $x\notin B$ or the assertion follows by (I2).
By (I3'), applied to $I+x$, $B$, and $x\in (I+x)\setminus B$ there is $y\in B\setminus (I + x)$ such that $I + y\in \I$.
\end{proof}

We proceed to prove \autoref{thm:i3'}. The following notation and terminology will be convenient. A circuit of $M$ which contains a given set $X\subseteq E(M)$ is called an \emph{$X$-circuit}.
 
By a \emph{representation of} a set $I\in\I(M_1\vee M_2)$, we mean a pair $(I_1, I_2)$ where $I_1\in\I(M_1)$ and $I_2\in \I(M_2)$ such that $I= I_1\cup I_2$.

For sets $I_1\in\I(M_1)$ and $I_2\in\I(M_2)$, and elements $x\in I_1\cup I_2$ and  $y\in E(M_1)\cup E(M_2)$ (possibly in $I_1\cup I_2$), a tuple $Y=(y_0=y, \ldots, y_n=x)$ is called an \emph{even $(I_1, I_2,y,x)$-exchange chain} (or \emph{even $(I_1, I_2,y,x)$-chain}) of \emph{length} $n$ if the following terms are satisfied. 

\begin{enumerate}[(X1)]
	\item For an even $i$, there exists a $\{y_i, y_{i+1}\}$-circuit $C_i\subseteq I_1+y_i$ of $M_1$.
	\item For an odd $i$, there exists a $\{y_i, y_{i+1}\}$-circuit $C_i\subseteq I_2+y_i$ of $M_2$.
\end{enumerate}

\noindent
If $n\geq 1$, then (X1) and (X2) imply that $y_0\notin I_1$ and that, starting with $y_1\in I_1\setminus I_2$, the elements $y_i$ alternate between $I_1\setminus I_2$ and $I_2\setminus I_1$; the single exception being $y_n$ which can lie in $I_1\cap I_2$.

By an \emph{odd exchange chain} (or \emph{odd chain}) we mean an even chain with the words `even' and `odd' interchanged in the definition.
Consequently, we say \emph{exchange chain} (or \emph{chain}) to refer to either of these notions.
Furthermore, a subchain of a chain is also a chain;
that is, given an $(I_1, I_2, y_0, y_n)$-chain $(y_0,\ldots, y_n)$, the tuple $(y_k,\ldots, y_l)$ is an $(I_1, I_2, y_k, y_l)$-chain for $0\leq k\leq l\leq n$.

\begin{lem}\label{thm:chain}
	If there exists an $(I_1, I_2,y,x)$-chain, then $(I + y) - x\in\I(M_1\vee M_2)$ where $I := I_1\cup I_2$.
	Moreover, if $x\in I_1\cap I_2$, then $I+y\in\I(M_1\vee M_2)$.
\end{lem}

\noindent
\textbf{Remark.}
In the proof of \autoref{thm:chain} chains are used in order to alter the sets $I_1$ and $I_2$; the change is in a single element. Nevertheless, to accomplish this change, exchange chain of arbitrary length may be required; for instance, a chain of length four is needed to handle the configuration depicted in \autoref{fig:chain}. 

\begin{figure}[htbp]
	\centering
	\subfloat[the initial representation\label{fig:chain1}]{\includegraphics{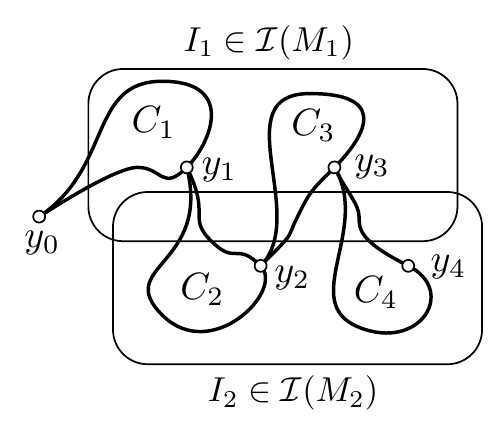}}
	\hspace{1cm}
	\subfloat[the obtained representation\label{fig:chain2}]{\includegraphics{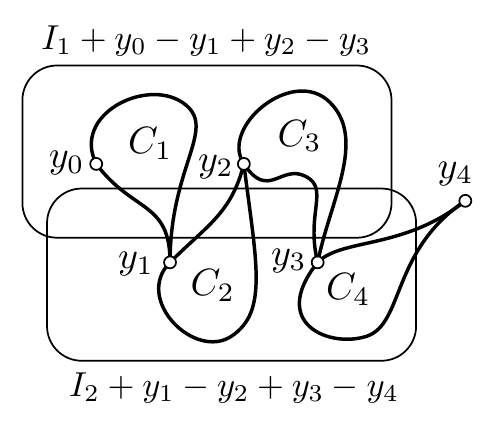}}
	\caption{An even exchange chain of length $4$.}
	\label{fig:chain}
\end{figure}

Next, we prove \autoref{thm:chain}.

\begin{proof}[Proof of \autoref{thm:chain}]
The proof is by induction on the length of the chain.
The statement is trivial for chains of length $0$.
Assume $n\geq 1$ and that $Y=(y_0,\ldots, y_n)$ is a shortest $(I_1, I_2, y, x)$-chain.
Without loss of generality, let $Y$ be an even chain.
If $Y':=(y_1,\ldots, y_n)$ is an (odd) $(I_1', I_2, y_1, x)$-chain where $I_1':= (I_1 + y_0) - y_1$, then $((I_1'\cup I_2) + y_1) - x\in\I(M_1\vee M_2)$ by the induction hypothesis and the assertion follows, since $(I_1'\cup I_2) + y_1 = (I_1\cup I_2) + y_0$.
If also $x\in I_1\cap I_2$, then either $x\in I_1'\cap I_2$ or $y_1 = x$ and hence $n=1$.
In the former case $I+y\in\I(M_1\vee M_2)$ follows from the induction hypothesis and in the latter case $I + y = I_1'\cup I_2\in\I(M_1\vee M_2)$ as $x\in I_2$.

Since $I_2$ has not changed, (X2) still holds for $Y'$, so to verify that $Y'$ is an $(I_1', I_2, y_1, x)$-chain, it remains to show $I_1'\in\I(M_1)$ and to check (X1).
To this end, let $C_i$ be a $\{y_i, y_{i+1}\}$-circuit of $M_1$ in $I_1 + y_i$ for all even~$i$.
Such exist by (X1) for $Y$.
Notice that any circuit of $M_1$ in $I_1 + y_0$ has to contain $y_0$ since $I_1\in\I(M_1)$.
On the other hand, two distinct circuits in $I_1+y_0$ would give rise to a circuit contained in $I_1$ by the circuit elimination axiom applied to these two circuits, eliminating $y_0$.
Hence $C_0$ is the unique circuit of $M_1$ in $I_1+y_0$ and $y_1\in C_0$ ensures $I_1'= (I_1 + y_0) - y_1\in\I(M_1)$.

To see (X1), we show that there is a $\{y_i, y_{i+1}\}$-circuit $C_i'$ of $M_1$ in $I_1' + y_i$ for every even $i\geq 2$.
Indeed, if $C_i\subseteq I_1' + y_i$, then set $C_i':= C_i$;
else, $C_i$ contains an element of $I_1\setminus I_1'=\{y_1\}$.
Furthermore, $y_{i+1}\in C_i\setminus C_0$;
otherwise $(y_0, y_{i+1}, \ldots, y_n)$ is a shorter $(I_1, I_2, y, x)$-chain for, contradicting the choice of $Y$.
Applying the circuit elimination axiom to $C_0$ and $C_i$, eliminating $y_1$ and fixing $y_{i+1}$, yields a circuit $C_i'\subseteq (C_0\cup C_i) - y_1$ of $M_1$ containing $y_{i+1}$.	
Finally, as $I_1'$ is independent and $C_i'\setminus I_1'\subseteq \{y_i\}$ it follows that $y_i\in C_i'$.
\end{proof}

We shall require the following. 
For $I_1\in\I(M_1)$, $I_2\in\I(M_2)$, and $x\in I_1\cup I_2$, let
\[
	A(I_1,I_2,x):= \{a\ |\ \text{there exists an $(I_1, I_2,a,x)$-chain}\}.
\]
This has the property that
\begin{equation}\label{eqn:a_property}
\begin{minipage}[c]{0.8\textwidth}
	for every $y\notin A$, either $I_1+y\in \I(M_1)$ or the unique circuit $C_y$ of $M_1$ in $I_1 + y$ is disjoint from $A$.
\end{minipage}
\end{equation}
To see this, suppose $I_1+y\notin\I(M_1)$.
Then there is a unique circuit $C_y$ of $M_1$ in $I_1 + y$.
If $C_y\cap A=\emptyset$, then the assertion holds so we may assume that $C_y\cap A$ contains an element, $a$ say.
Hence there is an $(I_1, I_2, a, x)$-chain $(y_0=a, y_1, \ldots, y_{n-1}, y_n=x)$.
As $a\in I_1$ this chain must be odd or have length $0$, that is, $a=x$.
Clearly, $(y, a, y_1,\ldots, y_{n-1}, x)$ is an even $(I_1, I_2, y, x)$-chain, contradicting the assumption that $y\notin A$.

Next, we prove \autoref{thm:i3'}.

\begin{proof}[Proof of \autoref{thm:i3'}]
Let $B\in\I(M_1\vee M_2)$ maximal, $I\in\I(M_1\vee M_2)$, and $x\in I\setminus B$.
Recall that we seek a $y\in B\setminus I$ such that $(I+y)-x\in\I(M_1\vee M_2)$.
Let $(I_1, I_2)$ and $(B_1, B_2)$ be representations of $I$ and $B$, respectively.
We may assume $I_1\in\B(M_1|I)$ and $I_2\in\B(M_2|I)$.
We may further assume that for all $y\in B\setminus I$ the sets $I_1+y$ and $I_2+y$ are dependent in $M_1$ and $M_2$, respectively, for otherwise it holds that $I+y\in\I(M_1\vee M_2)$ so that the assertion follows.
Hence, for every $y\in (B\cup I)\setminus I_1$ there is a circuit $C_y\subseteq I_1+y$ of $M_1$;
such contains $y$ and is unique since otherwise the circuit elimination axiom applied to these two circuits eliminating $y$ yields a circuit contained in $I_1$, a contradiction.

If $A:= A(I_1, I_2, x)$ intersects $B\setminus I$, then the assertion follows from \autoref{thm:chain}.
Else, $A\cap (B\setminus I) = \emptyset$, in which case we derive a contradiction to the maximality of $B$.
To this end, set (\autoref{fig:i_1_i_2})
\[
	B_1':= (B_1 \setminus b_1) \cup i_1\qquad\text{where}\quad b_1:= B_1\cap A\quad\text{and}\quad i_1:= I_1\cap A
\]
\[
	B_2':= (B_2 \setminus b_2) \cup i_2\qquad\text{where}\quad b_2:= B_2\cap A\quad\text{and}\quad i_2:= I_2\cap A
\]

\begin{figure} [htpb]   
	\begin{center}
		\includegraphics[width=5cm]{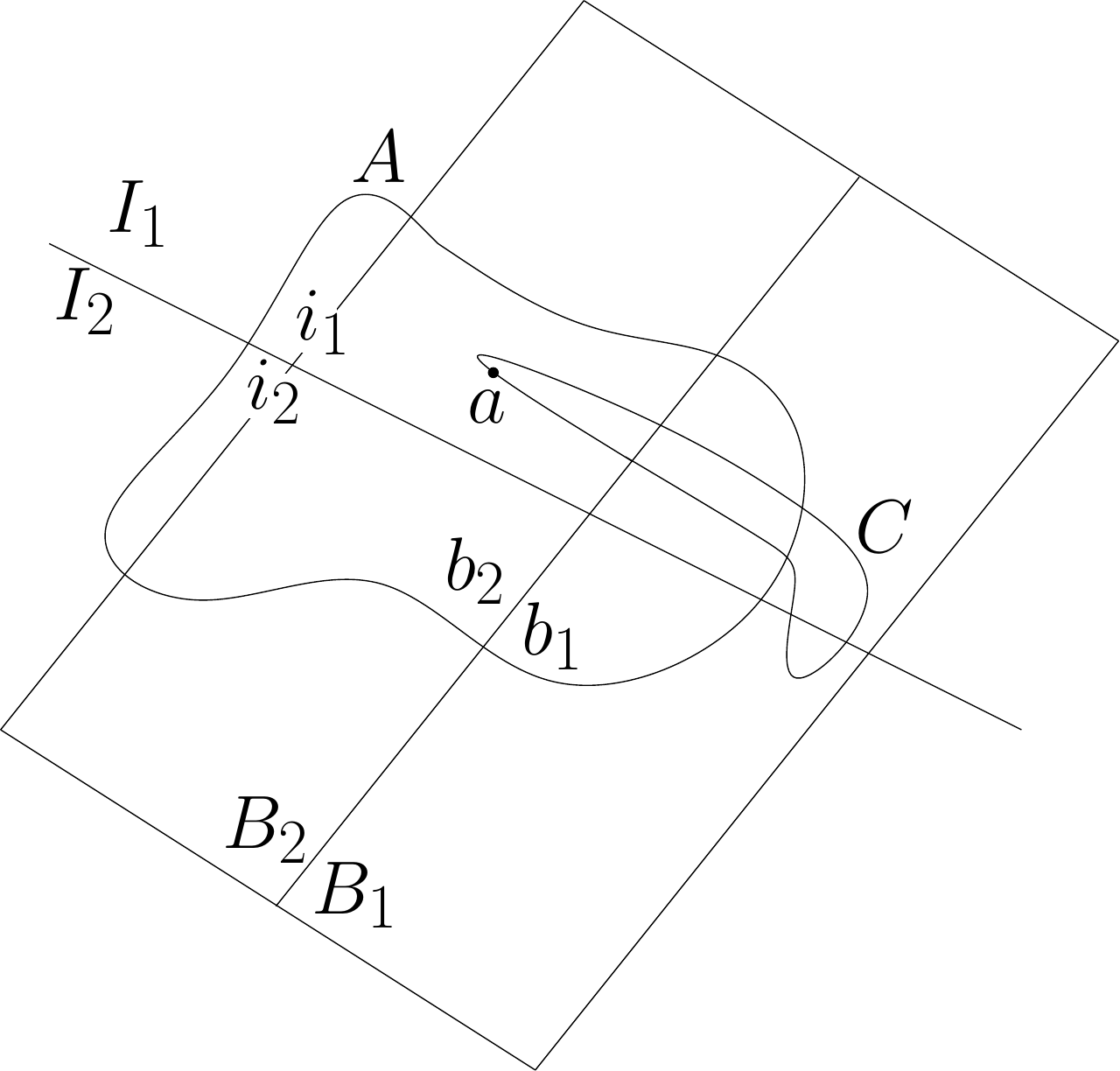}
		\caption{The independent sets $I_1$, at the top, and $I_2$, at the bottom, the bases $B_1$, on the right, and $B_2$, on the left, and their intersection with $A$.}
		\label{fig:i_1_i_2}
	\end{center}
\end{figure}

Since $A$ contains $x$ but is disjoint from $B\setminus I$, it holds that $(b_1\cup b_2) + x\subseteq i_1\cup i_2$ and thus $B+x \subseteq B_1'\cup B_2'$.
It remains to verify the independence of $B_1'$ and $B_2'$ in $M_1$ and $M_2$, respectively.

Without loss of generality it is sufficient to show $B_1'\in\I(M_1)$.
For the remainder of the proof `independent' and `circuit' refer to the matroid~$M_1$.
Suppose for a contradiction that the set $B_1'$ is dependent, that is, it contains a circuit~$C$.
Since $i_1$ and $B_1\setminus b_1$ are independent, neither of these contain~$C$.
Hence there is an element $a\in C\cap i_1\subseteq A$.
But $C\setminus I_1\subseteq B_1\setminus A$ and therefore no $C_y$ with $y\in C\setminus I_1$ contains $a$ by~(\ref{eqn:a_property}).
Thus, applying the circuit elimination axiom on $C$ eliminating all $y\in C\setminus I_1$ via $C_y$ fixing $a$, yields a circuit in $I_1$, a contradiction.
\end{proof}

Since in the proof of \autoref{thm:i3'} the maximality of $B$ is only used in order to avoid the case that $B+x\in\I(M_1\vee M_2)$, one may prove the following slightly stronger statement.

\begin{cor}
For all $I, J\in\I(M_1\vee M_2)$ and $x\in I\setminus J$, if $J+x\notin\I(M_1\vee M_2)$, then there exists $y\in J\setminus I$ such that $(I+y)-x\in\I(M_1\vee M_2)$.
\end{cor}

Next, the proof of \autoref{thm:i3'}, shows that for any maximal representation $(I_1,I_2)$ of $I$ there is $y\in B\setminus I$ such that exchanging finitely many elements of $I_1$ and $I_2$ gives a representation of $(I+y)-x$.

For subsequent arguments, it will be useful to note the following corollary. Above we used chains whose last element is fixed. One may clearly use chains whose first element is fixed. If so, then one arrives at the following.
  
\begin{cor}\label{thm:b2*} 
For all $I, J\in\I(M_1\vee M_2)$ and $y\in J\setminus I$, if $I+y\notin\I(M_1\vee M_2)$, then there exists $x\in I\setminus J$ such that $(I + y) - x\in\I(M_1\vee M_2)$.
\end{cor}

\subsection{Finitary matroid union}\label{sec:fin-mat}

In this section, we prove \autoref{thm:finitary-union}. In view of \autoref{thm:i3}, it remains to show that 
$\I(M_1\vee M_2)$ satisfies (IM) whenever $M_1$ and $M_2$ are finitary matroids. 

The verification of (IM) for countable finitary matroids can be done using  
K\"{o}nig's infinity lemma.
Here, in order to capture matroids on any infinite ground set, we employ a topological approach.
See~\cite{Armstrong} for the required topological background needed here.

We recall the definition of the product topology on $\P(E)$.
The usual base of this topology is formed by the system of all sets 
\[
C(A,B):= \{X\subseteq E\ |\ A\subseteq X, B\cap X=\emptyset\},
\]
where $A, B\subseteq E$ are finite and disjoint.
Note that these sets are closed as well.
Throughout this section, $\P(E)$ is endowed with the product topology and \emph{closed} is used in the topological sense only.\vspace{\baselineskip}

\noindent
We show that \autoref{thm:finitary-union} can easily be deduced from \autoref{thm:char_fini} and \autoref{thm:top_union}, presented next.

\begin{pro}\label{thm:char_fini}
Let  $\I= \lceil \I  \rceil\subseteq\P(E)$.
The following are equivalent.
\begin{enumerate}[\ref{thm:char_fini}.1. ]
	\item\label{fini_1} $\I$ is finitary;
	\item\label{fini_2} $\I$ is compact, in the subspace topology of $\P(E)$.
\end{enumerate}
\end{pro}

A standard compactness argument can be used in order to prove \ref{thm:char_fini}.\ref{fini_1}. Here, we employ a slightly less standard argument to prove \ref{thm:char_fini}.\ref{fini_2} as well.
Note that as $\P(E)$ is a compact Hausdorff space, assertion \ref{thm:char_fini}.\ref{fini_2} is equivalent to the assumption that $\I$ is closed in $\P(E)$, which we use quite often in the following proofs.

\begin{proof}[Proof of \autoref{thm:char_fini}]
To deduce \ref{thm:char_fini}.\ref{fini_2} from \ref{thm:char_fini}.\ref{fini_1}, we show that $\I$ is closed.
Let $X\notin \I$.
Since $\I$ is finitary, $X$ has a finite subset $Y\notin \I$ and no superset of $Y$ is in $\I$ as $\I=\lceil \I\rceil$.
Therefore, $C(Y,\emptyset)$ is an open set containing $X$ and avoiding $\I$ and hence $\I$ is closed.

For the converse direction, assume that $\I$ is compact and let $X$ be a set such that all finite subsets of $X$ are in $\I$.
We show $X \in \I$ using the finite intersection property\footnote{\emph{The finite intersection property} ensures that an intersection over a family $\cal C$ of closed sets is non-empty if every intersection of finitely many members of $\cal C$ is.} of $\P(E)$.
Consider the family ${\cal K}$ of pairs $(A,B)$ where $A\subseteq X$ and $B\subseteq E\setminus X$ are both finite.
The set $C(A,B)\cap \I$ is closed for every $(A,B)\in{\cal K}$, as $C(A,B)$ and $\I$ are closed.
If ${\cal L}$ is a finite subfamily of ${\cal K}$, then
\[
\bigcup_{(A,B)\in{\cal L}}A\in \bigcap_{(A,B)\in {\cal L}} \left(C(A,B)\cap\I\right).
\]

As $\P(E)$ is compact, the finite intersection property yields
\[
\left(\bigcap_{(A, B)\in{\cal K}} C(A,B)\right) \cap \I = \bigcap_{(A, B)\in{\cal K}} (C(A,B) \cap \I) \neq \emptyset.
\]
However, $\bigcap_{(A, B)\in{\cal K}} C(A,B)=\{X\}$. Consequently, $X \in \I$, as desired.
\end{proof}

\begin{lem}\label{thm:top_union}
If $\I$ and $\cal J$ are closed in $\P(E)$, then so is $\I\vee {\cal J}$.
\end{lem}

\begin{proof}
Equipping $\P(E) \times \P(E)$ with the product topology, yields that Cartesian products of closed sets in $\P(E)$ are closed in $\P(E) \times \P(E)$.
In particular, $\I \times {\cal J}$ is closed in $\P(E) \times \P(E)$. 
In order to prove that $\I \vee {\cal J}$ is closed, we note that $\I\vee {\cal J}$ is exactly the image of $\I \times {\cal J}$ under the union map
\[f: \P(E) \times \P(E) \to \P(E),\quad f(A,B)= A\cup B.\]
It remains to check that $f$ maps closed sets to closed sets;
which is equivalent to showing that $f$ maps compact sets to compact sets as $\P(E)$ is a compact Hausdorff space.
As continuous images of compact spaces are compact, it suffices to prove that $f$ is continuous, that is, to check that the pre-images of subbase sets $C(\{a\},\emptyset)$ and $C(\emptyset, \{b\})$ are open as can be seen here:
\begin{align*}
f^{-1}(C(\{a\},\emptyset)) &= \left(C(\{a\},\emptyset)\times \P(E)\right) \cup \left(\P(E)\times C(\{a\},\emptyset)\right)\\
f^{-1}(C(\emptyset,\{b\})) &= C(\emptyset,\{b\}) \times C(\emptyset,\{b\}).
\end{align*}
\end{proof}

Next, we prove \autoref{thm:finitary-union}.

\begin{proof}[Proof of \autoref{thm:finitary-union}.]
By \autoref{thm:i3} it remains to show that the union $\I(M_1) \vee \I(M_2)$ satisfies (IM). 
As all finitary set systems satisfy (IM), by Zorn's lemma, it is sufficient to show that $\I(M_1 \vee M_2)$ is finitary. 
By \autoref{thm:char_fini}, $\I(M_1)$ and $\I(M_2)$ are both compact and thus closed in $\P(E)$, yielding, by \autoref{thm:top_union}, that  $\I(M_1) \vee \I(M_2)$ is closed in $\P(E)$, and thus compact.
As $\I(M_1) \vee \I(M_2)=\lceil \I(M_1) \vee \I(M_2) \rceil$, \autoref{thm:char_fini} asserts that $\I(M_1) \vee \I(M_2)$ is finitary, as desired. 
\end{proof}

We conclude this section with the following observation.
\begin{obs}\label{exm:inf_union_fail}
A countable union of finitary matroids need not be a matroid.
\end{obs}

\begin{proof}
We show that for any integer $k\geq 1$, the set system 
$$
	\I :=\bigvee_{n\in\N} U_{k,\R}
$$
is not a matroid, where here $U_{k,\R}$ denotes the $k$-uniform matroid with ground set $\R$. 

Since a countable union of finite sets is countable, we have that the members of $\I$ are the countable subsets of $\R$. Consequently, the system $\I$ violates the (IM) axiom for $I=\emptyset$ and $X=\R$. 
\end{proof}

Above, we used the fact that the members of $\I$ are countable and that
the ground set is uncountable. One can have the following more subtle example, showing that a countable union of finite matroids need not be a matroid.
 
Let $A=\{a_1, a_2,\ldots\}$ and $B=\{b_1, b_2, \ldots\}$ be disjoint  countable sets, and for $n\in \N$, set $E_n:=\{a_1,\ldots, a_n\}\cup\{b_n\}$.
Then $\bigvee_{n\in\N} U_{1, E_n}$ is an infinite union of finite matroids and fails to satisfy (IM) for $I=A$ and $X = A\cup B = E(M)$.

\subsection{Nearly finitary matroid union}\label{sec:nearly-union}

In this section, we prove \autoref{thm:nearly-union}. 

For a matroid $M$, let $\I^{\fin}(M)$ denote the set of subsets of $E(M)$ containing no finite circuit of $M$, or equivalently, the set of subsets of $E(M)$ which have all their finite subsets in $\I(M)$.
We call $M^{\fin} = (E(M), \I^{\fin}(M))$ the \emph{finitarization} of $M$.
With this notation, a matroid $M$ is nearly finitary if it has the property that
\begin{equation}
\text{for each $J\in\I(M^{\fin})$ there exists an $I\in\I(M)$ such that $|J\setminus I| < \infty$.}
\end{equation}

For a set system $\I$ (not necessarily the independent sets of a matroid) we call a maximal member of $\I$ a \emph{base} and a minimal member subject to not being in $\I$ a \emph{circuit}.
With these conventions, the notions of \emph{finitarization} and \emph{nearly finitary} carry over to set systems.
\vspace{\baselineskip}

\noindent
Let $\I = \lceil\I\rceil$.
The finitarization $\I^{\fin}$ of $\I$ has the following properties.
\begin{enumerate}
	\item $\I\subseteq \I^{\fin}$ with equality if and only if $\I$ is finitary.
	\item $\I^{\fin}$ is finitary and its circuits are exactly the finite circuits of $\I$.
	\item $(\I|X)^{\fin} = \I^{\fin}|X$, in particular $\I|X$ is nearly finitary if $\I$ is.
\end{enumerate}
The first two statements are obvious.
To see the third, 
assume that $\I$ is nearly finitary and that $J\in \I^{\fin}|X\subseteq \I^{\fin}$.
By definition there is $I\in \I$ such that $J\setminus I$ is finite.
As $J\subseteq X$ we also have that $J\setminus (I\cap X)$ is finite and clearly $I\cap X\in\I|X$.

\begin{pro}\label{thm:finitarization}
The pair $M^{\fin}= (E,\I^{\fin}(M))$ is a finitary matroid, whenever $M$ is a matroid.
\end{pro}

\begin{proof}
By construction, the set system $\I^{\fin} = \I(M^{\fin})$ satisfies the axioms (I1) and (I2) and is finitary, implying that it also satisfies (IM).

It remains to show that $\I^{\fin}$ satisfies $(I3)$.
By definition, a set $X\subseteq E(M)$ is not in $\I^{\fin}$ if and only if it contains a finite circuit of $M$.

Let $B, I \in \I^{\fin}$ where $B$ is maximal and $I$ is not, and 
let $y\in E(M)\setminus I$ such that $I+y\in \I^{\fin}$.
If $I+x\in\I^{\fin}$ for any $x\in B\setminus I$, then we are done. 

Assuming the contrary, 
then $y\notin B$ and for any $x\in B\setminus I$ there exists a finite circuit $C_x$ of $M$ in $I+x$ containing $x$.
By maximality of $B$, there exists a finite circuit $C$ of $M$ in $B+y$ containing $y$.
By the circuit elimination axiom (in $M$) applied to the circuits $C$ and  $\{C_x\}_{x\in X}$ where $X:= C\cap (B\setminus I)$, there exists a circuit 
\[D\subseteq \left(C\cup\bigcup_{x\in X} C_x\right) \setminus X\subseteq I + y\]
of $M$ containing $y\in C\setminus \bigcup_{x\in X} C_x$.
The circuit $D$ is finite, since the circuits $C$ and $\{C_x\}$ are; this contradicts $I+y\in \I^{\fin}$.
\end{proof}

\begin{pro}\label{thm:fin-equal}
For arbitrary matroids $M_1$ and $M_2$ it holds that
\[{\cal I}(M^{\fin}_1 \vee M^{\fin}_2)= {\cal I}(M^{\fin}_1 \vee M^{\fin}_2)^{\fin}= {\cal I}(M_1 \vee M_2)^{\fin}.\]
\end{pro}

\begin{proof}
By \autoref{thm:finitarization}, the matroids $M_1^{\fin}$ and $M_2^{\fin}$ are finitary  and therefore $M^{\fin}_1 \vee M^{\fin}_2$ is a finitary as well, by \autoref{thm:finitary-union}.
This establishes the first equality.

The second equality follows from the definition of finitarization provided we show that the finite members of ${\cal I}(M^{\fin}_1 \vee M^{\fin}_2)$ and ${\cal I}(M_1 \vee M_2)$ are the same.

Since $\I(M_1)\subseteq \I(M_1^{\fin})$ and $\I(M_2)\subseteq \I(M_2^{\fin})$ it holds that ${\cal I}(M^{\fin}_1 \vee M^{\fin}_2)\supseteq {\cal I}(M_1 \vee M_2)$.
On the other hand, a finite set $I\in {\cal I}(M^{\fin}_1 \vee M^{\fin}_2)$ can be written as $I= I_1\cup I_2$ with $I_1\in {\cal I}(M^{\fin}_1)$ and $I_2\in {\cal I}(M^{\fin}_2)$ finite.
As $I_1$ and $I_2$ are finite, $I_1\in {\cal I}(M_1)$ and $I_2\in {\cal I}(M_2)$, implying that $I\in {\cal I}(M_1 \vee M_2)$.
\end{proof}

With the above notation a matroid $M$ is nearly finitary if each base of $M^{\fin}$ contains a base of $M$ such that their difference is finite. The following is probably the most natural manner to construct nearly finitary matroids (that are not finitary) from finitary matroids.

For a matroid $M$ and an integer $k\geq 0$, set $M[k]:= (E(M), \I[k])$, where
\[\I[k]:= \{I\in \I(M)\ |\ \exists J \in \I(M)\text{ such that }I \subseteq J\text{ and }|J \setminus I| = k\}.\]

\begin{pro}\label{thm:fin-to-nearly}
If $\rank(M)\geq k$, then $M[k]$ is a matroid.
\end{pro}

\begin{proof}
The axiom (I1) holds as $\rank(M)\geq k$; the axiom (I2) holds as it does in $M$.
For (I3) let $I',I\in \I(M[k])$ such that $I'$ is maximal and $I$ is not.
There is a set $F'\subseteq E(M)\setminus I'$ of size $k$ such that, in $M$, the set $I'\cup F'$ is not only independent but, by maximality of $I'$, also a base.
Similarly, there is a set $F\subseteq E(M)\setminus I$ of size $k$ such that $I\cup F\in\I(M)$.

We claim that $I\cup F$ is non-maximal in $\I(M)$ for any such $F$.
Suppose not and $I\cup F$ is maximal for some $F$ as above.
By assumption, $I$ is contained in some larger set of $\I(M[k])$.
Hence there is a set $F^+\subseteq E(M)\setminus I$ of size $k+1$ such that $I\cup F^+$ is independent in $M$.
Clearly $(I\cup F) \setminus (I\cup F^+) = F\setminus F^+$ is finite, so \autoref{thm:sym-dif} implies that
\[\left|F^+\setminus F\right| = \left|(I\cup F^+)\setminus (I\cup F)\right|\leq \left|(I\cup F) \setminus (I\cup F^+)\right| = \left|F\setminus F^+\right|.\]
In particular, $k +1 = |F^+|\leq |F| = k$, a contradiction.

Hence we can pick $F$ such that $F\cap F'$ is maximal and, as $I\cup F$ is non-maximal in $\I(M)$, apply (I3) in $M$ to obtain a $x\in (I'\cup F') \setminus (I\cup F)$ such that $(I\cup F) + x \in \I(M)$.
This means $I + x\in\I(M[k])$.
And $x\in I'\setminus I$ follows, as $x\notin F'$ by our choice of $F$.

To show (IM), let $I\subseteq X\subseteq E(M)$ with $I\in \I(M[k])$ be given.
By (IM) for $M$, there is a $B\in\I(M)$ which is maximal subject to $I\subseteq B\subseteq X$.
We may assume that $F:=B\setminus I$ has at most $k$ elements;
for otherwise there is a superset $I'\subseteq B$ of $I$ such that $|B\setminus I'|=k$ and it suffices to find a maximal set containing $I'\in\I(M[k])$ instead of $I$.

We claim that for any $F^+\subseteq X\setminus I$ of size $k+1$ the set $I\cup F^+$ is not in $\I(M[k])$.
For a contradiction, suppose it is.
Then in $M|X$, the set $B = I\cup F$ is a base and $I\cup F^+$ is independent and as $(I\cup F)\setminus (I\cup F^+)\subseteq F\setminus F^+$ is finite, \autoref{thm:sym-dif} implies
\[\left|F^+\setminus F\right| = \left|(I\cup F^+)\setminus (I\cup F)\right|\leq \left|(I\cup F)\setminus (I\cup F^+)\right| = \left|F\setminus F^+\right|.\]
This means $k +1 = |F^+|\leq |F| = k$, a contradiction.
So by successively adding single elements of $X\setminus I$ to $I$ as long as the obtained set is still in $\I(M[k])$ we arrive at the wanted maximal element after at most $k$ steps.
\end{proof}

We conclude this section with a proof of \autoref{thm:nearly-union}. To this end,
we shall require following two lemmas. 

\begin{lem}\label{thm:sym-dif}
Let $M$ be a matroid and $I, B\in\I(M)$ with $B$ maximal and $B\setminus I$ finite. Then, $|I\setminus B| \leq |B\setminus I|$.
\end{lem} 

\begin{proof}
The proof is by induction on $|B\setminus I|$.
For $|B\setminus I| = 0$ we have $B\subseteq I$ and hence $B = I$ by maximality of $B$.
Now suppose there is $y\in B\setminus I$.
If $I+y\in \I$ then by induction
\[|I\setminus B| = |(I+y)\setminus B|\leq |B\setminus (I+y)| = |B\setminus I| - 1\]
and hence $|I\setminus B| < |B\setminus I|$.
Otherwise there exists a unique circuit $C$ of $M$ in $I+y$.
Clearly $C$ cannot be contained in $B$ and therefore has an element $x\in I\setminus B$.
Then $(I+y)-x$ is independent, so by induction
\[|I\setminus B|  - 1 = |((I+y)- x)\setminus B|\leq |B\setminus ((I+y) - x)| = |B\setminus I| - 1,\]
and hence $|I\setminus B| \leq |B\setminus I|$.
\end{proof}

\begin{lem}\label{thm:nearlyfinitaryim}
	Let $\I\subseteq \P(E)$ be a nearly finitary set system satisfying (I1), (I2), and the following variant of (I3):
	\begin{itemize}
		\item[(*)] For all $I, J\in\I$ and all $y\in I\setminus J$ with $J+y\notin\I$ there exists $x\in J\setminus I$ such that $(J + y) - x\in\I$.
	\end{itemize}
	Then $\I$ satisfies (IM).
\end{lem}

\begin{proof}
Let $I\subseteq X\subseteq E$ with $I\in\I$.
As $\I^{\fin}$ satisfies (IM) there is a set $B^{\fin}\in\I^{\fin}$ which is maximal subject to $I\subseteq B^{\fin}\subseteq X$ and being in $\I^{\fin}$.
As $\I$ is nearly finitary, there is $J\in\I$ such that $B^{\fin}\setminus J$ is finite and we may assume that $J\subseteq X$.
Then, $I\setminus J\subseteq B^{\fin}\setminus J$ is finite so that we may choose a $J$ minimizing $|I\setminus J|$. 
If there is a $y\in I\setminus J$, then by (*) we have $J+y\in \I$ or there is an $x\in J\setminus I$ such that $(J+y)-x\in\I$.
Both outcomes give a set containing more elements of $I$ and hence contradicting the choice of $J$.

It remains to show that $J$ can be extended to a maximal set $B$ of $\I$ in~$X$.
For any superset $J'\in\I$ of $J$, we have $J'\in\I^{\fin}$ and $B^{\fin}\setminus J'$ is finite as it is a subset of $B^{\fin}\setminus J$.
As $\I^{\fin}$ is a matroid, \autoref{thm:sym-dif} implies
\[|J'\setminus B^{\fin}|\leq|B^{\fin}\setminus J'| \leq |B^{\fin}\setminus J|.\]
Hence, $|J'\setminus J| \leq 2|B^{\fin}\setminus J| < \infty$.
Thus, we can greedily add elements of $X$ to $J$ to obtain the wanted set $B$ after finitely many steps.
\end{proof}

Next, we prove \autoref{thm:nearly-union}. 

\begin{proof}[Proof of \autoref{thm:nearly-union}]
By \autoref{thm:i3'}, in order to prove that $M_1\vee M_2$ is a matroid, it is sufficient to prove that $\I(M_1\vee M_2)$ satisfies (IM).
By \autoref{thm:b2*} and \autoref{thm:nearlyfinitaryim} it remains to show that $\I(M_1\vee M_2)$ is nearly finitary.

So let $J\in \I(M_1 \vee M_2)^{\fin}$.
By \autoref{thm:fin-equal} we may assume that $J=J_1\cup J_2$ with $J_1\in \I(M_1^{\fin})$ and $J_2\in \I(M_2^{\fin})$.
By assumption there are $I_1\in \I(M_1)$ and $I_2\in\I(M_2)$ such that $J_1\setminus I_1$ and $J_2\setminus I_2$ are finite.
Then $I=I_1\cup I_2\in\I(M_1 \vee M_2)$ and the assertion follows as $J\setminus (I_1 \cup I_2)\subseteq (J_1\setminus I_1) \cup (J_2\setminus I_2)$ is finite.
\end{proof}

\subsubsection{Unions of non-nearly finitary matroids}\label{sec:non-nearly}

In this section, we prove \autoref{thm:tightness} asserting that a certain family of non-nearly finitary matroids does not admit a union theorem.

A matroid $N$ is non-nearly finitary provided it has a set $I\in\I(N^{\fin})$ with the property that no finite subset of $I$ meets all the necessarily infinite circuits of $N$ in $I$.
If we additionally assume that there is one such $I$ which contains only countably many circuits, then there exists a finitary matroid $M$ such that $\I(M\vee N)$ is not a matroid.

\begin{proof}[Proof of \autoref{thm:tightness}]
For $N$ and $I$ as in \autoref{thm:tightness} choose an enumeration $C_1, C_2, \ldots$ of the circuits of $N$ in $I$.
We may assume that $I = \bigcup_{n\in\N} C_n$.
There exist countably many disjoint subsets $Y_1, Y_2,\ldots$ of $I$ satisfying
\begin{enumerate}
	\item $|Y_n|\leq n$ for all $n\in\N$; and
	\item $Y_n\cap C_i\neq\emptyset$ for all $n\in\N$ and all $1\leq i\leq n$.
\end{enumerate}

We construct the above sets as follows. 
Suppose $Y_1, \ldots, Y_n$ have already been defined.
Let $Y_{n+1}$ be a set of size at most $n+1$ disjoint to each of $Y_1, \ldots, Y_n$ and meeting the circuits $C_1, \ldots, C_{n+1}$;
such exists as $\bigcup_{i=1}^n Y_i$ is finite and all circuits in $I$ are infinite.

Let $L = \{l_1, l_2, \ldots\}$ be a countable set disjoint from $E(N)$.
For each $n\in\N$ let $M_n$ be the $1$-uniform matroid on $Y_n\cup \{l_n\}$, i.e.\ $M_n := U_{1, Y_n\cup \{l_n\}}$.
Then, $M := \bigoplus_{n\in\N} M_n$ is a direct sum of finite matroids and hence finitary.

We contend that $I\in \I(M\vee N)$ and that $\I(M\vee N)$ violates (IM) for $I$ and $X := I\cup L$.
By construction, $Y_n$ contains some element $d_n$ of $C_n$, for every $n\in\N$.
So that $J_M = \{d_1, d_2, \ldots\}$ meets every circuit of $N$ in $I$ and is independent in $M$.
This means that $J_N := I\setminus J_M \in\I(N)$ and thus $I = J_M\cup J_N\in\I(M\vee N)$.

It is now sufficient to show that a set $J$ satisfying  $I\subseteq J\subseteq X$ is in $\I(M \vee N)$ if and only if it misses infinitely many elements $L'\subseteq L$.
Suppose that $J\in \I(M\vee N)$.
There are sets $J_M\in\I(M)$ and $J_N\in\I(N)$ such that $J=J_M\cup J_N$.
As $D := I\setminus J_N$ meets every circuit of $N$ in $I$ by independence of $J_N$, the set $D$ is infinite.
But $I\subseteq J$ and hence $D\subseteq J_M$.
Let $A$ be the set of all integers $n$ such that $Y_n\cap D\neq\emptyset$.
As $Y_n$ is finite for every $n\in\N$, the set $A$ must be infinite and so is $L' := \{l_n\ |\ n\in A\}$.
Since $J_M$ is independent in $M$ and any element of $L'$ forms a circuit of $M$ with some element of $J_M$, we have $J_M\cap L'=\emptyset$ and thus $J\cap L'=\emptyset$ as no independent set of $N$ meets~$L$.

Suppose that there is a sequence $i_1 < i_2 < \ldots$ such that $J$ is disjoint from $L'=\{l_{i_n}\ |\ n\in\N\}$.
We show that the superset $X\setminus L'$ of $J$ is in $\I(M\vee N)$.
By construction, for every $n\in\N$, the set $Y_{i_n}$ contains an elements $d_n$ of $C_n$.
Set $D:= \{d_n\ |\ n\in\N\}$.
Then $D$ meets every circuit of $N$ in $I$, so $J_N := I\setminus D$ is independent in $N$.
On the other hand, $D$ contains exactly one element of each $M_n$ with $n\in L'$.
So $J_M := (L\setminus L') \cup D\in\I(M)$ and therefore $X\setminus L' = J_M\cup J_N\in\I(M\vee N)$.
\end{proof}

It is not known wether or not the proposition remains true if we drop the requirement that there are only countable many circuits in $I$.

\section{Base packing in co-finitary matroids}\label{sec:cover-pack}

In this section, we prove \autoref{thm:pack}, which is a base packing theorem for co-finitary matroids. 

\begin{proof}[Proof of \autoref{thm:pack}]
As the `only if' direction is trivial, it remains to show the `if' direction. 
For a matroid $N$ and natural numbers $k,c$ put
\[
 \I[N,k,c]:=\{X\subseteq E({N})\mid \exists I_1,...,I_k \in \I(N)\text{ with } g_c(I_1,...,I_k)=X\},
\]
where $  g_c(I_1,...,I_k):=\{e:  |\{j: e\in I_j\}|\geq c\}$.
The matroid $M$ has $k$ disjoint spanning sets if and only if $M^*$ has $k$ independent sets such that every element of $E$ is in at least $k-1$ of those independent sets.
Put another way, $M$ has $k$ disjoint bases if and only if  
\begin{equation}
	\I[M^*,k,k-1]=\P(E).
\end{equation}

As $M^*$ is finitary, $\I[M^*,k,k-1]$ is finitary by an argument similar to that in the proof of \autoref{thm:top_union}; 
here one may define
\[
f: \P(E)^k \to \P(E); \\ f(A_1,...,A_k)= g_{k-1}(A_1,...,A_k),
\]
and repeat the above argument.

Thus, it suffices to show that every finite set $Y$ is in $\I[M^*,k,k-1]$.
To this end, it is sufficient to find $k$ independent sets of $M^*$ such that every element of $Y$ is in at least $k-1$ of those;
complements of which are $M$-spanning sets $S_1,...,S_k$ such that these are disjoint if restricted to $Y$.
To this end, we show that there are disjoint spanning sets $S_1',...,S_k'$ of $M.Y$ and set $S_i:=S_i'\cup (E-Y)$.
Since \autoref{thm:pack} is true for finite matroids~\cite{Oxley}, the sets $S_1',..., S_k'$ exist if and only if $|Z|\geq k \cdot \rank_{M.Y}(Y|Y-Z)$ for all $Z\subseteq Y$.
As $|Z|\geq k \cdot \rank(E|E-Z)$, by assumption, and as $\rank(E|E-Z)=  \rank_{M.Y}(Y|Y-Z)$~\cite[Lemma 3.13]{matroid_axioms}, the assertion follows.
\end{proof}

It might be worth noting that this proof easily extends to arbitrary finite families of co-finitary matroids.

Finally, we use \autoref{thm:finitary-union} (actually only the fact that (IM) is satisfied for unions of finitary matroids), to derive a base covering result for finitary matroids. The finite base covering theorem asserts that {\sl a finite matroid $M$ can be covered by $k$ bases if and only if $\rank(X) \geq |X| /k$ for every $X \subseteq E(M)$}~\cite{schrijverBook}. 

\begin{cor}\label{thm:cover}
A finitary matroid $M$ can be covered by $k$ independent sets if and only if $\rank_M(X)\geq |X|/k$ for every finite $X \subseteq E(M)$.  
\end{cor}

This claim is false if $M$ is an infinite circuit, implying that this result is best possible in the sense that $M$ being finitary is necessary.

\begin{proof}
The `only if' implication is trivial.
Suppose then that each finite set $X \subseteq E(M)$ satisfies $\rank_M(X) \geq |X|/ k$ and put $N = \bigvee_{i=1}^k M$;
such is a finitary matroid by \autoref{thm:finitary-union}.
If $N$ is the free matroid, the assertion holds trivially. 
Suppose then that $N$ is not the free matroid and consequently contains a circuit $C$;
such is finite as $N$ is finitary.
Hence, $M|C$ cannot be covered by $k$ independent sets of $M|C$ so that by the finite matroid covering theorem~\cite[Theorem 12.3.12]{Oxley} there exists a finite set $X \subseteq C$ such that $\rank_{M|C}(X) < |X|/k$ which clearly implies $\rank_M(X) < |X|/k$;
a contradiction.    
\end{proof}

\bibliographystyle{plain}
\bibliography{literatur}

\end{document}